\documentclass[a4paper,12pt]{article}
\usepackage[margin=1.3in]{geometry}  
\usepackage{hyperref}
\usepackage{breakurl}
\usepackage{comment}
\usepackage{microtype}
\usepackage{marginnote}

\usepackage{graphicx}              
\usepackage{amsmath}
\usepackage{amscd}               
\usepackage{amsfonts} 
\usepackage{amssymb}             
\usepackage{amsthm}                
\usepackage{mathrsfs}
\usepackage{amsbsy}
\usepackage{mhequ}
\usepackage{tikz}
\usepackage{wasysym}
\usepackage{centernot}

\newcommand{\be}{\begin{equation}}
\newcommand{\ee}{\end{equation}}
\newcommand{\bea}{\begin{eqnarray}}
\newcommand{\eea}{\end{eqnarray}}

\numberwithin{equation}{section}

\newtheorem{thm}{Theorem}[section]
\newtheorem{defn}[thm]{Definition}
\newtheorem{lem}[thm]{Lemma}

\newtheorem{cor}[thm]{Corollary}

\newtheorem{eg}[thm]{Example}

\makeatletter
\def\th@newremark{\th@remark\thm@headfont{\bfseries}}
\makeatletter

\theoremstyle{newremark}

\definecolor{darkgreen}{rgb}{0,0.5,0}
\definecolor{darkred}{rgb}{0.7,0,0}
\definecolor{darkblue}{rgb}{0,0,0.7}
\newcommand{\aA}{\mathcal{A}}

\newcommand{\cC}{\mathcal{C}}
\newcommand{\dD}{\mathcal{D}}

\newcommand{\iI}{\mathcal{I}}

\newcommand{\lL}{\mathcal{L}}
\newcommand{\mM}{\mathcal{M}}
\newcommand{\nN}{\mathcal{N}}

\newcommand{\pP}{\mathcal{P}}

\newcommand{\sS}{\mathcal{S}}

\newcommand{\wW}{\mathcal{W}}

\def\emptyset{\mathop{\centernot\ocircle}}

\newcommand{\E}{\mathbf{E}}

\newcommand{\R}{\mathbf{R}}

\newcommand{\Z}{\mathbf{Z}}
\newcommand{\0}{\mathbf{0}}

\newcommand{\n}{\mathbf{n}}

\newcommand{\z}{\mathbf{z}}

\newcommand*\Bell{\ensuremath{\boldsymbol\ell}}

\newcommand*\Bpsi{\ensuremath{\boldsymbol\psi}}
\newcommand*\Bzeta{\ensuremath{\boldsymbol\zeta}}
\newcommand*\Btheta{\ensuremath{\boldsymbol\theta}}

\relax
\relax

\colorlet{symbols}{blue!90!black}
\colorlet{testcolor}{green!60!black}

\def\${|\!|\!|}

\def\P{\mathbf{P}}

\usetikzlibrary{shapes.misc}
\usetikzlibrary{shapes.symbols}
\usetikzlibrary{snakes}
\usetikzlibrary{decorations}
\usetikzlibrary{decorations.markings}

\def\drawx{\draw[-,solid] (-3pt,-3pt) -- (3pt,3pt);\draw[-,solid] (-3pt,3pt) -- (3pt,-3pt);}
\tikzset{
	root/.style={circle,fill=testcolor,inner sep=0pt, minimum size=2mm},
	dot/.style={circle,fill=black,inner sep=0pt, minimum size=1mm},
	var/.style={circle,fill=black!10,draw=black,inner sep=0pt, minimum size=
		2mm},
	bvar/.style={circle,fill=black!15,draw=white,inner sep=0pt, minimum size=
		8mm},
	dotred/.style={circle,fill=black!50,inner sep=0pt, minimum size=2mm},
	generic/.style={semithick,shorten >=1pt,shorten <=1pt},
	dist/.style={ultra thick,draw=testcolor,shorten >=1pt,shorten <=1pt},
	testfcn/.style={ultra thick,testcolor,shorten >=1pt,shorten <=1pt,<-},
	testfcnx/.style={ultra thick,testcolor,shorten >=1pt,shorten <=1pt,<-,
		postaction={decorate,decoration={markings,mark=at position 0.6 with {\drawx}}}},
	kepsilon/.style={semithick,shorten >=1pt,shorten <=1pt,densely dashed,->},
	kprimex/.style={semithick,shorten >=1pt,shorten <=1pt,densely dashed,->,
		postaction={decorate,decoration={markings,mark=at position 0.4 with {\drawx}}}},
	kernel/.style={semithick,shorten >=1pt,shorten <=1pt,->},
	multx/.style={shorten >=1pt,shorten <=1pt,
		postaction={decorate,decoration={markings,mark=at position 0.5 with {\drawx}}}},
	kernelx/.style={semithick,shorten >=1pt,shorten <=1pt,->,
		postaction={decorate,decoration={markings,mark=at position 0.4 with {\drawx}}}},
	kernel1/.style={->,semithick,shorten >=1pt,shorten <=1pt,postaction={decorate,decoration={markings,mark=at position 0.45 with {\draw[-] (0,-0.1) -- (0,0.1);}}}},
	kernel2/.style={->,semithick,shorten >=1pt,shorten <=1pt,postaction={decorate,decoration={markings,mark=at position 0.45 with {\draw[-] (0.05,-0.1) -- (0.05,0.1);\draw[-] (-0.05,-0.1) -- (-0.05,0.1);}}}},
	kernelBig/.style={semithick,shorten >=1pt,shorten <=1pt,decorate, decoration={zigzag,amplitude=1.5pt,segment length = 3pt,pre length=2pt,post length=2pt}},
	gepsilon/.style={dotted,semithick,shorten >=1pt,shorten <=1pt},
	renorm/.style={shape=circle,fill=white,inner sep=1pt},
	labl/.style={shape=rectangle,fill=white,inner sep=1pt},
	xi/.style={circle,fill=symbols!10,draw=symbols,inner sep=0pt,minimum size=1.2mm},
	xix/.style={crosscircle,fill=symbols!10,draw=symbols,inner sep=0pt,minimum size=1.2mm},
	xib/.style={circle,fill=symbols!10,draw=symbols,inner sep=0pt,minimum size=1.6mm},
	xibx/.style={crosscircle,fill=symbols!10,draw=symbols,inner sep=0pt,minimum size=1.6mm},
	not/.style={circle,fill=symbols,draw=symbols,inner sep=0pt,minimum size=0.5mm},
	>=stealth,
}

\makeatletter
\def\DeclareSymbol#1#2#3{\expandafter\gdef\csname MH@symb@#1\endcsname{\tikz[baseline=#2,scale=0.15,draw=symbols]{#3}}\expandafter\gdef\csname MH@symb@#1s\endcsname{\scalebox{0.7}{\tikz[baseline=#2,scale=0.15,draw=symbols]{#3}}}}
\def\<#1>{\csname MH@symb@#1\endcsname}
\makeatletter

\DeclareSymbol{Xi22}{0.5}{\draw (0,0) node[xi] {} -- (-1,1) node[not] {} -- (0,2) node[xi] {}; }

\begin{document}

\title{Signature inversion for monotone paths}
\author{Jiawei Chang\footnote{University of Oxford, UK. Email: jiawei.chang@maths.ox.ac.uk}, Nick Duffield\footnote{Texas A\&M University, USA. Email: duffieldng@tamu.edu}, Hao Ni\footnote{University College London, UK. Email: h.ni@ucl.ac.uk}, Weijun Xu\footnote{University of Warwick, UK. Email: weijun.xu@warwick.ac.uk}}

\maketitle

\begin{abstract}
	The aim of this article is to provide a simple sampling procedure to reconstruct any monotone path from its signature. For every $N$, we sample a lattice path of $N$ steps with weights given by the coefficient of the corresponding word in the signature. We show that these weights on lattice paths satisfy the large deviations principle. In particular, this implies that the probability of picking up a ``wrong" path is exponentially small in $N$. The argument relies on a probabilistic interpretation of the signature for monotone paths. 
\end{abstract}

\section{Introduction}

\subsection{The signature of a path}

A path $\gamma$ is a continuous map from a fixed interval $J$ into a metric space $(V, \|\cdot\|_{V})$. The length of $\gamma$ is defined by
\begin{align*}
\|\gamma\| := \sum_{\dD(J)} \| \gamma(u_{j+1}) - \gamma(u_j) \|_{V}, 
\end{align*}
where the sum is taken over all dissections $\dD(J) = \{u_{j}\}$ of the interval $J$. $\gamma$ is said have bounded variations if $\|\gamma\| < +\infty$. In what follows, we consider $V = \R^{d}$. Note that although the choice of the norm on $\R^{d}$ affects the actual length of $\gamma$, it does not affect whether $\gamma$ has bounded variations or not. 

Let $\{e_1, \dots, e_d\}$ denote the standard basis of $\R^{d}$. For every integer $n \geq 0$, a word of length $n$ is an ordered sequence of $n$ letters from the set $\{e_1, \dots, e_d\}$ (with repetition allowed). We use $|w|$ to denote the length of the word $w$, that is, the number of letters consisting of the word. For two words $w_1 = e_{i_1} \dots e_{i_n}$ and $w_2 = e_{j_1} \dots e_{j_m}$, their concatenation $w_1 * w_2$ is a new word of length $n+m$ defined by
\begin{align*}
w_{1} * w_{2} = e_{i_1} \dots e_{i_n} e_{j_1} \dots e_{j_m}. 
\end{align*}
We use $\emptyset$ to denote the empty word, which is the unique word of length $0$. The signature of a bounded variations path is defined as follows. 

\begin{defn} \label{de:signature}
	Let $\gamma: [0,1] \rightarrow \R^{d}$ be a continuous path of bounded variations. For every integer $n \geq 0$ and every word $w = e_{i_1} \dots e_{i_n}$, let
	\begin{align*}
	C_{\gamma}(w) = \int_{0 < u_{1} < \cdots < u_{n} < 1} \dot{\gamma}^{i_1}(u_1) \cdots \dot{\gamma}^{i_n}(u_n) d u_1 \cdots d u_n, 
	\end{align*}
	where $\gamma^{i}$ is the $i$-th component of $\gamma$. The signature of $\gamma$ is the formal power series
	\begin{align*}
	X(\gamma) = \sum_{n=0}^{+\infty} \sum_{|w|=n} C_{\gamma}(w) \cdot w, 
	\end{align*}
	where the second sum is taken over all words of length $n$, and we have set $C_{\gamma}(\emptyset) = 1$ by convention. 
\end{defn}

We call the collection $\{C_{\gamma}(w): |w|=n\}$ the $n$-th level coefficients in the signature. The signature is a definite integral over the fixed interval where $\gamma$ is defined. Changing the parametrisation or the size of the interval does not change the signature of $\gamma$. The signature contains important information about the path. For example, the collection $\{C_{\gamma}(w): |w|=1\}$ reproduces the increment of the path, and the second level coefficients $\{C_{\gamma}(w): |w|=2\}$ represents the (signed) areas enclosed by the projection of the path on $e_i - e_j$ planes. 

It was proved by Hambly and Lyons (\cite{sig_uniqueness}) that bounded variation paths are uniquely determined by their signatures up to tree-like pieces.\protect\footnote{Roughly speaking, two paths $\alpha$ and $\beta$ are tree-equivalent if the path $\alpha * \beta^{-1}$, obtained by running $\alpha$ first and then $\beta$ backwards, is a ``null-path" in the sense that all the trajectories cancel out themselves. Please refer to \cite{sig_tree} for a precise definition.} In \cite{symmetrisation}, Lyons and one of the authors developed a procedure based on the use of symmetrisation that enables one to reconstruct every $\cC^{1}$ path (when at natural parametrisation) from its signature. The purpose of this article is to give a significant simplification of the reconstruction procedure in the case when $\gamma$ is monotone.

\subsection{Monotone paths and the main result}

From now on, we fix our path $\gamma: [0,1] \rightarrow \R^{d}$ that is monotone. We can assume without loss of generality that $\gamma$ is monotonically increasing such that
\begin{align*}
\dot{\gamma}^{i}(t) \geq 0, \qquad \forall t \in [0,1], \forall i = 1, \dots, d. 
\end{align*}
We also equip $\R^{d}$ with the $\ell^{1}$ norm, so the length of a monotone path is then simply the sum of all its first level coefficients in the signature. Thus, we can assume without loss of generality that $L=1$; otherwise one can just simply recover $L$ first and rescale the path by $L^{-1}$ so that the new path has length $1$. Finally, since the signature is invariant under re-parametrisation, we also assume $\gamma$ is at natural parametrisation so that
\begin{equation} \label{eq:unit_speed}
\sum_{i=1}^{d} \dot{\gamma}^{i}(u) = 1, \qquad \forall u \in [0,1]. 
\end{equation}
Since $\gamma$ is monotonically increasing, we have $C_{\gamma}(w) \geq 0$ for every word $w$, and for every integer $N$, we have
\begin{align*}
\sum_{|w|=N} C_{\gamma}(w) = \frac{L^{N}}{N!} = \frac{1}{N!}
\end{align*}
since we have assumed $L=1$. This suggests that for every $N$, the quantities $\{ N! C_{\gamma}(w): |w|=N \}$ constitute a probability measure on the words of length $N$, giving each word $w$ with $|w|=N$ the ``probability" $N! C(w)$. Now, for every word with length $N$, we associate to it a lattice path $X_{N}$ with step size $\frac{1}{N}$ such that $X_{N}$ is a monotone lattice path parametrised at unit speed, and moves in exactly the same direction as the word $w$. More precisely, if
\begin{align*}
w = e_{i_1} \cdots e_{i_N}, 
\end{align*}
then
\begin{align*}
X_{N} = \frac{1}{N} \big( e_{i_1} * \cdots * e_{i_N} \big)
\end{align*}
at natural parametrisation, where ``*" denotes the concatenation of two paths, and we have had an abuse use of the notation $e_{i_k}$ also to denote the one-step lattice path moving in the $e_{i_k}$ direction. Now, for every $N \geq 0$, we assign the $N$-step paths $\{X_{N}: |w|=N\}$ ``probabilities" $N! C(w)$. This gives us a sequence of laws on the space of lattice paths. The main result of our article is the following. 

\begin{thm} \label{th:main_loose}
	The laws on $\{X_{N}\}$ above satisfies a large deviations principle on the space of continuous function from $[0,1]$ to $\R^{d}$. 
\end{thm}

The above theorem implies that one can reconstruct any monotone path from its signature by sampling directly from the lattice paths with weights given by the corresponding terms in the signature. More precisely, for fixed large $N$, one ``sample" a lattice path $\sS_{N} w$ according to the ``probabilities" $\{C(w): |w|=N\}$. The large deviations principle for these laws in Theorem \ref{th:main_loose} then ensures that the chance of picking a wrong lattice path is exponentially small in $N$. 

The proof of this theorem relies on a probabilistic interpretation of
the signature of monotone paths. Once this observation is made, the
rest follows directly from standard large deviations techniques. Also,
unfortunately, the rate function for the LDP for $\{X_N\}$ does not
have a closed form. However, an observation by \cite{functional_LDP}
suggests that we can add another random process $T_{N}$ (to be defined
below) to $\{X_{N}\}$, so that the pair $(X_{N}, T_{N})$ satisfies LDP
with a rate function of closed form. These will be made more precise in Section \ref{sec:lattice} below.

\begin{flushleft}
	\textbf{Acknowledgements}
\end{flushleft}
WX thanks Terry Lyons for very helpful discussions, especially for telling him about the probabilistic interpretation of signatures for monotone paths. ND is an Associate Member of the Oxford-Man Institute of Quantitative Finance, which partially supported this work with a visitorship during July 2015. WX is supported by EPSRC through the research fellowship EP/N021568/1.

\section{Sampling path large deviations}

\subsection{Probabilistic interpretation}
We first give the probabilistic interpretation of the signature of a monotone path. Let $\gamma: [0,1] \rightarrow \R^{d}$ be a monotone path parametrised at unit speed in the sense of \eqref{eq:unit_speed}, and it has length $1$ under this assumption. We can associate it with a probability measure on random lattice paths in the following way. Consider $d$ independent Poisson processes run simultaneously on the time interval $[0,1]$, generating letters $e_{1}, \dots, e_{d}$, respectively. Let $W(t)$ be the word of ordered letters that arrive up to time $t$. For example, if at times $0 < u_{1} < \dots < u_{5} < t$, the letters $e_{3}, e_{2}, e_{2}, e_{1}, e_{3}$ arrives, then
\begin{align*}
W(t) = e_{3} e_{2} e_{2} e_{1} e_{3}, \qquad W(v) = e_{2} e_{2}, \quad v \in [u_{1}, u_{2}). 
\end{align*}
One can make $W(\cdot)$ into a lattice path in the following way. Suppose the arrival times are $\tau_{j}$ for $j=1, 2, \dots$ with the convention that $\tau_{0} = 0$, then $W$ can be defined as a lattice path by setting $W(\tau_{0}) = 0$, and
\begin{align*}
W(t) = W(\tau_{j}) + (t-\tau_{j}) e_{i_j}, 
\end{align*}
where $e_{i_j}$ is the arriving letter at time $\tau_{j}$. We have thus associated to $\gamma$ a random lattice path $W$. 

We will be interested in the laws of $W$ conditional on the total number of arrivals up to time $1$. Thus, we let $\nN(t)$ be the process counting the total number of arrivals up to time $t$. Since $\gamma$ is parametrised at unit speed, $\nN(t)$ is a homogeneous Poisson process on $[0,1]$ with intensity $1$. 

Now, we condition on the event $\nN(1) = N$, that is, there are totally $N$ arrivals up to time $1$ (when the path runs out). Let
\begin{align*}
\P^{N}(\cdot) = \P \big(\cdot | \nN(1) = N \big)
\end{align*}
denote the conditional probability. Thus, for every word $w$ with $|w|=N$, with the abuse use of notation that $W$ denoting the word generated by the processes, we have precisely the relation
\begin{equation} \label{eq:prob_interpretation}
C(w) = \frac{1}{N!} \P^{N} \big( W = w \big) = \frac{1}{N!} \P \big( W = w | \nN(1) = N \big). 
\end{equation}
This is the probabilistic interpretation of the signature for lattice paths.

\subsection{Lattice path sampling and large deviations}
\label{sec:lattice}

From now on, we always condition on $\nN(1) = N$, and we will prove large deviations for a sequence of conditional laws $\lL \big( \cdot | \nN(1) = N\big)$. Let $W(t)$ be the random lattice path generated by the conditional Poisson process, and $W_{N}(t) = \frac{1}{N} \wW(t)$. Thus, under $\P^{N}$, every realisation of $W_{N}$ is a lattice path with step size $\frac{1}{N}$ and total length $1$ ($N$ steps). 

Our aim is to show the large deviations principles for the processes $W_{N}$ and its random time change version. For this, we first introduce the proper function spaces that the processes live in. Let $\cC$ denote the set of continuous functions from $[0,1]$ to $\R$ with the uniform topology, and let $\cC^{d}$ be $d$ copies of $\cC$. Also, we let
\begin{align*}
\aA = \bigg\{ \psi: [0,1] \rightarrow \R \phantom{1} \text{s.t.} \phantom{1} \psi(0)=0,  \phantom{1} \psi \phantom{1} \text{continuous} \phantom{1} \text{and} \phantom{1} \text{non-decreasing} \bigg\}, 
\end{align*}
and
\begin{align*}
\aA^{d}_{a} = \bigg\{ \Bpsi = (\psi^{1}, \dots, \psi^{d}): \psi^{i} \in \aA, \sum_{i=1}^{d} \psi^{i}(1) = a \bigg\}. 
\end{align*}
Note that $\psi \in \aA$ implies $\psi$ is absolutely continuous with $\dot{\psi} \geq 0$. For the set $\aA^{d}_{a}$, we will mainly use it for $a=1$. Also, we let the function $I: \R^{+} \times \R^{+} \rightarrow \R$ be
\begin{equation} \label{eq:function}
I(x,y) = x \log(x/y), \qquad x,y > 0. 
\end{equation}
Our first aim is to show that the conditional laws
\begin{align*}
\mu_{N} = \lL \big( W(\cdot) | \nN(1) = N \big)
\end{align*}
obey a large deviations principle for processes. In order to derive the rate function, we need the following lemma for its finite dimensional approximations.

\begin{lem} \label{le:finite}
	Let $k \geq 1$. For every $\{0 = u_{0} < u_{1} < \dots < u_{k} = 1\}$, the conditional laws
	\begin{align*}
	\lL \big( W_{N}(u_1) - W_{N}(u_0), \dots, W_{N}(u_k) - W_{N}(u_{k-1}) | \nN(1) = N \big)
	\end{align*}
	satisfy the large deviations principle with scale $N$ and good rate function
	\begin{align*}
	\iI_{k}(\z) = \sum_{i=1}^{d} \sum_{j=1}^{k} (u_j - u_{j-1}) \cdot I \bigg( \frac{z_{j}^{i} - z_{j-1}^{i}}{u_{i} - u_{i-1}}, \frac{\gamma_{u_j}^{i} - \gamma_{u_{j-1}}^{i}}{u_{j} - u_{j-1}} \bigg)
	\end{align*}
	when $\z = (\z_{1}, \dots,\z_{k}) \in (\R^{d})^{k}$ satisfies $\z_{j} \geq 0$, and
	\begin{align*}
	\sum_{i=1}^{d} z_{k}^{i} = 1, \qquad z_{j}^{i} \leq z_{j+1}^{i}, \qquad \forall i = 1, \dots, d, \phantom{1} j = 1, \dots, k, 
	\end{align*}
    where we have used the convention that $\z_{0} = \0$. Otherwise $\iI_{k}(\z) = \infty$. 
\end{lem}
\begin{proof}
	Fix $k \geq 1$ and $0 = u_{0} < \cdots < u_{k} = 1$, and let $\Lambda_{N}$ be the log moment generating function of the multi-vector $\big( \wW_{N}(u_1) - \wW_{N}(u_0), \dots, \wW_{N}(u_k) - \wW_{N}(u_{k-1}) \big)$, conditioned on $\nN_1 = N$. Here, each component is a $d$-dimensional vector, and this should be understood as a random vector in $(\Z / N)^{dk}$. Then, the conditional distribution (on $\nN_1 = N$) of this random vector is precisely multinomial with $N$ trials and probabilities
	\begin{align*}
	p_{j}^{i} = \gamma_{u_j}^{i} - \gamma_{u_{j-1}}^{i}. 
	\end{align*}
	Here, $p_{i}^{i}$ denotes the probability of the outcome of the trial being ``$\wW_{N}^{i}(u_j) - \wW_{N}^{i}(u_{i-1})$". Also, the $p_j^i$'s are already normalised as a probability since we have assumed $L=1$. Hence, for $\Btheta = \{\theta_j^i\} \in \R^{dk}$, we have
	\begin{align*}
	\frac{1}{N} \Lambda_{N}(N \Btheta) &= \frac{1}{N} \log \E \exp \bigg( \sum_{i,j} \theta_{j}^{i} \big( W_{N}^{i}(u_j) - W_{N}^{i}(u_{j-1}) \big) \bigg)\\ 
	&= \log \bigg( \sum_{i,j} \big( \gamma_{u_j}^{i} - \gamma_{u_{j-1}}^{i} \big) e^{\theta_{j}^{i}} \bigg), 
	\end{align*}
	where the second equality follows from the moment generating function for multinomial distribution, and the sum is taken over the set
	\begin{align*}
	i \in \{1, \dots, d\}, \qquad j \in \{1, \dots, k\}. 
	\end{align*} 
	Hence, the sequence $\{\Lambda_n\}$ satisfies the assumption of G\"{a}rtner-Ellis theorem (\cite[Theorem 2.3.6]{DZ}). Let
	\begin{align*}
	\Lambda(\Btheta) := \log \bigg(\sum_{i,j} \big(\gamma_{u_j}^{i} - \gamma_{u_{j-1}}^{i} \big) e^{\theta_{j}^{i}} \bigg), 
	\end{align*}
	then the laws 
	\begin{align*}
	\lL \big( W_{N}(u_1) - W_{N}(u_0), \dots, W_{N}(u_k) - W_{N}(u_{k-1}) | \nN(1) = N \big)
	\end{align*}
	satisfy the large deviations principle with the rate function
	\begin{align*}
	\Lambda^{*}(\z) &:= \sup_{\Btheta \in \R^{dk}} \bigg(\sum_{i,j} \theta_{j}^{i} \big(z_{u_j}^{i} - z_{u_{j-1}}^{i} \big) - \Lambda(\Btheta) \bigg) \\
	&= \sum_{i,j} \big( z_{u_j}^{i} - z_{u_{j-1}}^{i} \big) \log \bigg( \frac{z_{u_j}^{i} - z_{u_{j-1}}^{i}}{\gamma_{u_j}^{i} - \gamma_{u_{j-1}}^{i}} \bigg), 
	\end{align*}
	which is precisely the stated form. 
\end{proof}

We are now ready to give the large deviations principle for the conditional laws on the rescaled paths $W_N$. 

\begin{thm} \label{th:ldp1}
	For every $N \geq 0$, let
	\begin{align*}
	\mu_{N} = \lL \big( W_{N} | \nN(1) = N \big)
	\end{align*}
	be the law on $\cC^{d}$. Then, $\{\mu_{N}\}$ satisfies a large deviations principle with scale $N$ and good rate function
	\begin{align*}
	\iI_{W}(\Bpsi) = \sum_{i=1}^{d} \int_{0}^{1} I \big( \dot{\psi}^{i}(t), \dot{\gamma}^{i}(t) \big) dt
	\end{align*}
	when $\Bpsi \in \aA^{d}_{1}$, and $\iI_{W}(\Bpsi) = \infty$ otherwise. 
\end{thm}
\begin{proof}
	By Lemma \ref{le:finite}, any finite dimensional distribution of difference of the processes $\wW_N$ satisfies the large deviations principle. Thus, by \cite[Theorem 1]{process_LDP}, the laws of the processes $\wW_N(\cdot)$'s also satisfy the large deviations principle with good rate function
	\begin{align*}
	\iI_{W}(\Bpsi) = \sum_{i=1}^{d} \int_{0}^{1} I \big(\dot{\psi}^{i}(u), \dot{\gamma}^{i}(u) \big) du, \quad \psi^{i} \in \aA_{0} \phantom{1} \text{s.t.} \phantom{1} \sum_{i=1}^{d} \psi^{i}(1) = 1, 
	\end{align*}
	and $\infty$ otherwise. 
\end{proof}

Note that the above large deviations principle are for the processes $\{W_{N}(\cdot)\}$, where the time parametrisation is random and cannot be observed from the signature. We thus need to parametrise the paths $\wW_N$'s at unit speed. For this reason, we introduce the random time change below. 

We still condition on $\nN(1) = N$. For $j = 1, \dots, N$, let $\tau_{j} \in [0,1]$ denote the arrival time of the $j$-th word in the process $W$, so we have
\begin{align*}
\nN(t) = j, \qquad t \in [\tau_{j}, \tau_{j+1}). 
\end{align*}
Let the random map $T_{N}: [0,1] \rightarrow [0,\tau_N]$ be such that
\begin{align*}
T_{N}(q) = \tau_{j} + (Nq-j)(\tau_{j+1}-\tau_{j}), \qquad q \in \bigg( \frac{j}{N}, \frac{j+1}{N} \bigg]. 
\end{align*}
This says $T_{N}(j/N)$ is the arrival time of the $j$-th word, and linearly interpolate in between. Thus, $T_{N}$ is almost surely a strictly increasing map with inverse $Q_{N}: [0, \tau_N] \rightarrow [0,1]$ defined by
\begin{align*}
Q_{N}(t) = \frac{j}{N} + \frac{t - \tau_{j}}{N(\tau_{j+1} - \tau_{j})}, \qquad t \in \big( \tau_{j}, \tau_{j+1} \big]. 
\end{align*}
Thus, for every realisation such that $\nN(1) = N$, the random path $W_N \circ T_N$ is parametrised at unit speed. The intuition is that the map 
\begin{align*}
W_N \circ T_N: [0,1] \rightarrow (\Z/N)^{d}, \qquad q \mapsto W_{N} (T_N(q))
\end{align*}
takes the $(Nq)$-th arrival of the $N$ letters to the position ``$q$" of the path $W_N$. We let
\begin{align*}
X_{N} = W_{N} \circ T_{N}, 
\end{align*}
which is nothing but the lattice path $W_{N}$ re-parametrised at unit speed. To investigate the LDP for $X_N$, note that $T_{N} = Q_{N}^{-1}$, and the operations
\begin{align*}
(W_{N}, Q_{N}) \mapsto (W_{N}, T_{N}) \mapsto W_{N} \circ T_{N}
\end{align*}
are both continuous. Thus, by contraction principle, it suffices to prove the LDP for $(W_{N}, Q_{N})$. We give it in the following lemma. 

\begin{lem} \label{le:ldp_time}
	The laws for $(W_N, Q_N) \in \cC^{d} \times \cC$ conditioned on $\nN(1) = N$ satisfy the large deviations principle with scale $N$ and good rate function
	\begin{align*}
	\iI_{(W,Q)}(\Bpsi, \phi) = \sum_{i=1}^{d} \int_{0}^{1} I \big( \dot{\psi}^{i}(t), \dot{\gamma}^{i}(t) \big) dt
	\end{align*}
	for $\Bpsi \in \aA^{d}_{1}$, $\phi \in \aA$ such that $\sum_{i} \psi^{i} = \phi$, and $\iI_{(W,Q)} = \infty$ otherwise. 
\end{lem}
\begin{proof}
	The definition of $Q_{N}$ ensures that
	\begin{align*} 
	Q_{N}(t) = \sum_{i} W_{N}^{i}(t)
	\end{align*}
	for every $t \in [0,1]$, so the rate function for the pair $(W,Q)$ is the same as $\iI_{W}$ except that one further requires $\sum_{i} \psi^{i} = \phi$. Note that together with $\Bpsi \in \aA^{d}_{1}$, this constraint forces $\phi(1) = 1$ in order for $\iI_{(W,Q)}$ to be finite. 
\end{proof}

\begin{cor}
	The laws $\lL (X_{N} | \nN(1) = N)$ on $\cC^{d}$ satisfies a large deviations principle. 
\end{cor}

The rate function for $X$ can be expressed in terms of $\iI_{(W,Q)}$
using the Contraction Principle \cite{DZ} but it does not have a closed form. A nice observation from \cite{functional_LDP} suggests that we can add the component $T_{N}$ so that the pair $(X_{N}, T_{N})$ satisfies the large deviations principle with a closed form rate function. This is the content of the following theorem. 

\begin{thm} \label{th:ldp2}
	The conditional laws
	\begin{align*}
	\nu_{N} = \lL \big( (X_N, T_N) | \nN(1) = N \big)
	\end{align*}
	satisfy a large deviations principle with scale $N$ and good rate function
	\begin{equation} \label{eq:ldp2}
	\iI_{(X, T)} \big(\Bzeta, \xi \big) = \iI_{(W,Q)} \big( \Bzeta \circ \xi^{-1}, \xi^{-1} \big) = \sum_{i=1}^{d} \int_{0}^{1} I \big( \dot{\zeta}^{i}(q), (\gamma^{i} \circ \xi)'(q) \big) dq
	\end{equation}
	for $\zeta^{i}, \xi \in \aA$ such that $\sum_{i} \dot{\zeta}^{i} \equiv 1$, and $\iI_{X, T} = \infty$ otherwise. 
\end{thm}
\begin{proof} 
	Since $T_{N} = Q_{N}^{-1}$, it follows directly from the large deviations for inverse processes (see e.g. \cite[Theorem 4]{functional_LDP}) and Lemma \ref{le:ldp_time} that
	\begin{align*}
	\iI_{(X, T)} \big(\Bzeta, \xi \big) = \iI_{(W,Q)} \big( \Bzeta \circ \xi^{-1}, \xi^{-1} \big) = \sum_{i=1}^{d} \int_{0}^{1} I \bigg( \big( \zeta^{i} \circ \xi^{-1} \big)'(t), \dot{\gamma}^{i}(t) \bigg) dt. 
	\end{align*}
	To derive the specific form of the rate function (the second equality in \eqref{eq:ldp2}), we note that for each $i$, we have
	\begin{align*}
	\big( \zeta^{i} \circ \xi^{-1} \big)'(t) = \frac{\dot{\zeta}^{i}(\xi^{-1}(t))}{\dot{\xi}(\xi^{-1}(t))}, 
	\end{align*}
	so a change of variable $q = \xi^{-1}(t)$ gives
	\begin{align*}
	\int_{0}^{1} I \bigg( \big( \zeta^{i} \circ \xi^{-1} \big)'(t), \dot{\gamma}^{i}(t) \bigg) dt &= \int_{0}^{1} I \bigg( \frac{\dot{\zeta}^{i} \big( \xi^{-1}(t) \big)}{\dot{\xi}\big( \xi^{-1}(t) \big)}, \dot{\gamma}^{i}(t) \bigg) dt\\
	&= \int_{0}^{1} I \bigg( \frac{\dot{\zeta}^{i}(q)}{\dot{\xi}(q)}, \dot{\gamma}^{i}(q) \bigg) d \xi(q)\\
	&= \int_{0}^{1} \dot{\zeta}^{i}(q) \log \bigg( \frac{\dot{\zeta}^{i}(q)}{ \dot{\gamma}^{i}(\xi(q)) \dot{\xi}(q)} \bigg) dq. 
	\end{align*}
	The constraint that $\xi \in \aA_{0,1}$ is obvious. For the constraint of $\zeta$, we notice that by Lemma \ref{le:ldp_time}, we need
	\begin{align*}
	\sum_{i=1}^{d} \zeta^{i}(\xi^{-1}(q)) = \xi^{-1}(q), \qquad \forall q \in [0,1]. 
	\end{align*}
	This is equivalent as $\sum_{i} \psi^{i}(t) = t$ for all $t$, or $\sum_{i} \dot{\psi}^{i} \equiv 1$. 
\end{proof}

\subsection{Connections with the symmetrisation procedure}
\label{sec:symmetrisation}

In \cite{symmetrisation}, the authors used a procedure of symmetrisation to produce a deterministic sequence of piecewise linear approximations from the signature of a $\cC^{1}$ path to the true path. The construction of that sequence requires rather complicated operations beyond symmetrisation between terms in the signatures. The aim of this subsection is to show that, in the case of monotone paths, these piecewise linear paths can also be "sampled" in a rather straightforward way after symmetrisation. This is actually a simple consequence of the large deviations principle for lattice path sampling. 

We first briefly recall the symmetrisation procedure on signatures used in \cite{symmetrisation}, and will mainly follow the notations there. For every integer $N \geq 0$ and $k \geq 0$, let $\pP_{N,k}$ denote the set of $k$-partitions of $N$; that is, 
\begin{align*}
\pP_{N,k} = \bigg\{ \n = \big( n_{1}, \dots, n_{k} \big): n_{j} \geq 0, \phantom{1} \sum_{j=1}^{k} n_{j} = N \bigg\}. 
\end{align*}
For $\n \in \pP_{N,k}$, let
\begin{align*}
\lL^{\n}_{k} = \bigg\{ \Bell = \big( \Bell_{1}, \dots, \Bell_{k} \big): \Bell_{j} = \big( \ell_{j}^{1}, \dots, \ell_{j}^{d} \big), \phantom{1} \sum_{i=1}^{d} \ell_{j}^{i} = n_{j}, \phantom{1} \forall j = 1, \dots, k \bigg\}. 
\end{align*}
Now, for $\n \in \pP_{N,k}$ and $\Bell \in \lL^{\n}_{k}$, define the set of words $\wW^{\n}_{k}(\Bell)$ by
\begin{align*}
\wW^{\n}_{k}(\Bell) = \bigg\{ w = w_{1} * \cdots * w_{k}: |w_{j}|_{e_i} = \ell_{j}^{i}, \phantom{1} \forall i=1, \dots, d, j = 1, \dots, k \bigg\}, 
\end{align*}
and define the symmetrised signatures by
\begin{align*}
\sS^{\n}_{k}(\Bell) := N! \sum_{w \in \wW^{\n}_{k}(\Bell)} C(w). 
\end{align*}
In other words, $\sS^{\n}_{k}(\Bell)$ is the sum of the coefficients of all words $w$ such that $w = w_{1} * \cdots * w_{k}$, and the number of letters $e_{i}$ in $w_{j}$ is $\ell_{j}^{i}$. 

Note that in \cite{symmetrisation}, the symmetrisation procedure is taken with $n_j \equiv n$ for all $j$, and $N = kn$, so the set-up above is a slight generalisation of that in \cite{symmetrisation}. Recall the random word $W$ generated by the Poisson process associated to the path $\gamma$; we have
\begin{align*}
\P^{N} \big( W \in \wW^{\n}_{k} \big) = \sS^{\n}_{k}(\Bell). 
\end{align*}
Thus, each $\wW^{\n}_{k}$ corresponds to a random piecewise linear path, which we call $Y^{\n}_{k}$. We have the following theorem. 

\begin{thm} \label{th:ldp_symmetrisation}
	For every $N \geq 0$, let $k = k(N)$ and $\n \in \pP_{N,k(N)}$ be such that
	\begin{align*}
	k(N) \rightarrow +\infty, \quad \text{and} \quad \sup_{1 \leq j \leq k(N)} \frac{n_{j}}{N} \rightarrow 0
	\end{align*}
	as $N \rightarrow +\infty$. Then, the sequence $(Y^{\n}_{k,N}, T_n)$ satisfies the large deviations principle with the same rate function as in Theorem \ref{th:ldp2}. 
\end{thm}
\begin{proof}
	It suffices to show that $Y^{\n}_{k,N}$ and $X_{N} \circ T_N$ are exponentially equivalent. In fact, for every realisation of the lattice path $X_N \circ T_N$, $Y^{\n}_{k}$ is its polygonal approximation such that its $j$-th piece connects the points $X_{N} \big( T_{N}(\bar{n}_{j-1} / N) \big)$ and $X_{N} \big( T_{N}(\bar{n}_{j}/N) \big)$, where
	\begin{align*}
	\bar{n}_{j} = \sum_{\ell=1}^{j} n_{\ell}. 
	\end{align*}
	Thus, the difference between the $j$-th piece of $Y^{\n}_{k}$ and the corresponding part in $X_N \circ T_N$ is at most $\frac{n_j}{N}$, and hence
	\begin{align*}
	\| Y^{\n}_{k}(q) - X_{N}(q) \big) \|_{\infty} \leq \sup_{j} \frac{n_{j}}{N}.  
	\end{align*}
	Thus, we have
	\begin{align*}
	\P^{N} \big( \|Y^{\n}_{k} - X_{N} \|_{\infty} \geq \delta \big) \leq \P^{N} \bigg( \sup_{j} \frac{n_j}{N} \geq \delta \bigg) = 0
	\end{align*}
	for all sufficiently large $N$. This proves the exponential equivalence of $(Y^{\n}_{k}, T_N)$ and $(X_{N}, T_{N})$, and hence the LDP follows. 
\end{proof}

\section{A numerical example}

In this section, we will provide a numerical example for the sampling scheme introduced above. For the convenience of simulation and display, we use piecewise linear approximations as in Section \ref{sec:symmetrisation}, and we use $n_j \equiv n$ for all $j$, so the total number of arrivals is $N = kn$. We study the following example. 

\begin{eg} \label{eg:example}
	Let $\gamma: [0, 1] \rightarrow \R^{2}$ be the path
	\begin{align*}
		\gamma(t) = t e_{1} + t^{2} e_{2} = (t, t^{2}), \qquad \forall t \in [0, 1].
	\end{align*}
	Thus, $\gamma$ is a two-dimensional monotone path with length $1$. 
\end{eg}

We compute the truncated signature of the piecewise linear approximation of $X$ with the time mesh $0.01$ as the approximate value of that of $X$. Then we apply the inversion algorithm outlined in the above section, and we obtain the following results.

\subsection{$k=2$, \quad $3 \leq n \leq 8$}

We now test the two-piece approximation to $\gamma$ in Example \ref{eg:example}. We fix $k=2$, and vary $n$ from $3$ to $8$ to look at their accuracies. We show the probability matrices $\mM_{k,n}$ for the two-piece approximations for different $n$ in the tables below. 

The entries in the $j$-th row ($j=1,2$) represents the ``probabilities" that are assigned to various directions of the $j$-th linear piece. More precisely, the $(j,m)$ entry is the ``probability" that the direction of the $j$-th piece has the $e_1:e_2$ ratio $m: (n-m)$. The red colour indicates which of the direction has been assigned the biggest weight in sampling. The weight of the two-piece linear path is simply the product of the weights of two pieces, and these two pieces have the same $\ell^{1}$ length. 

For example, in Table \ref{table:2_4}, the ``probability" that the direction of the first piece has $e_{1}: e_{2}$ ratio $m: (4-m)$ for $m = 0, 1, 2, 3, 4$ are
\begin{align*}
2.64 \times 10^{-2}, \quad 1.51 \times 10^{-1}, \quad 3.35 \times 10^{-1}, \quad 3.46 \times 10^{-1}, \quad 1.42 \times 10^{-1}, 
\end{align*}
respectively, and the direction with the ratio $3:1$ has the biggest weight. For $n=4$, the maximum weight of the two-piece linear piece is the one whose $e_{1}: e_{2}$ ratios of the two pieces are $3:1$ and $1:3$, respectively. 

The two-piece linear paths with the biggest weight for $n = 3, \dots, 8$ are plotted in Figure \ref{figure:k2}. One can see clearly that the MLS estimator of the path is closer to the true path $\gamma(t) = (t, t^{2})$ as $t$ increases. 

\begin{table}[ht]
	\centering
	\begin{tabular}{c | c cc c}
		\hline
		$j \backslash m$ &0 & 1&2&3\\
		\hline
		1&6.63E-02&2.83E-01&	\color{red}{4.24E-01}	&2.26E-01\\
		2&2.21E-01	&\color{red}{4.31E-01}	&2.84E-01	&6.32E-02\\
		\hline
	\end{tabular}
	\caption{$k=2, n=3$}
	\label{table:2_3}
\end{table}

\begin{table}[ht]
	\centering
	\begin{tabular}{c | c cc c c}
		\hline
		$j \backslash m$ &0 & 1&2&3&4\\
		\hline
		1&2.64E-02 & 1.51E-01 & 3.35E-01 & \color{red}{3.46E-01} & 1.42E-01  \\
		2&1.37E-01 & \color{red}{3.51E-01} & 3.40E-01 & 1.48E-01 & 2.44E-02  \\
		\hline
	\end{tabular}
	\caption{$k=2, n=4$}
	\label{table:2_4}
\end{table}

\begin{table}[ht]
	\centering
	\begin{tabular}{c | c cc cc c}
		\hline
		$j \backslash m$ &0 & 1&2&3&4&5\\
		\hline
		1&1.05E-02&	7.51E-02&	2.21E-01&	\color{red}{3.37E-01}&	2.67E-01&	8.94E-02\\
		2& 8.45E-02	&2.69E-01	&\color{red}{3.44E-01}&	2.22E-01	&7.20E-02	&9.44E-03\\
		\hline
	\end{tabular}
	\caption{$k=2, n=5$}
	\label{table:2_5}
\end{table}

\begin{table}[ht]
	\centering
	\begin{tabular}{c | c c c c }
		\hline
		$j \backslash m$  &0 &1&2&3\\
		\hline
		1&4.17E-03&	3.59E-02&	1.31E-01&	\color{red}{2.64E-01}\\
		2&5.23E-02&	1.98E-01&	\color{red}{3.15E-01}&	2.68E-01		\\
		\hline
		\hline
		$j \backslash m$& 4 &5&6\\
		\hline
		1   	&4.17E-03&	3.59E-02	&1.31E-01	&2.64E-01\\
		2& 5.23E-02&	1.98E-01&	3.15E-01	&2.68E-01\\
		\hline
	\end{tabular}
	\caption{$k=2, n=6$}
	\label{table:2_6}
\end{table}

\begin{table}[ht]
	\centering
	\begin{tabular}{c | c c c c  }
		\hline
		$j \backslash m$  &0 &1&2&3\\
		\hline
		1&1.66E-03	&1.67E-02	&7.31E-02	&1.82E-01\\
		2&3.24E-02	&1.42E-01	&2.70E-01	&\color{red}{2.86E-01}\\
		\hline
		\hline
		$j \backslash m$&4 &5&6&7\\
		\hline
		1   	&\color{red}{2.80E-01}&	2.66E-01&	1.45E-01&	3.54E-02\\
		2&1.83E-01&	7.04E-02&	1.52E-02&	1.41E-03\\
		\hline
	\end{tabular}
	\caption{$k=2, n=7$}
	\label{table:2_7}
\end{table}

\begin{table}[ht]
	\centering
	\begin{tabular}{c | c c c c c c c }
		\hline
		$j \backslash m$  &0 &1&2&3&4\\
		\hline
		1&8.18E-04	&8.58E-03	&4.11E-02	&1.18E-01&2.19E-01	\\
		2&1.99E-02	&9.94E-02	&2.19E-01	&\color{red}{2.77E-01}&2.22E-01		\\
		\hline
		\hline
		$j \backslash m$  &5&6&7&8\\
		\hline
		1   	&\color{red}{2.71E-01}	&2.17E-01	&1.03E-01	&2.22E-02\\
		2    &1.16E-01	&3.87E-02	&7.64E-03	&6.81E-04\\
		\hline
	\end{tabular}
	\caption{$k=2, n=8$}
	\label{table:2_8}
\end{table}

\clearpage

\begin{figure}[!ht]
	\centering
	\includegraphics[width=0.8\textwidth]{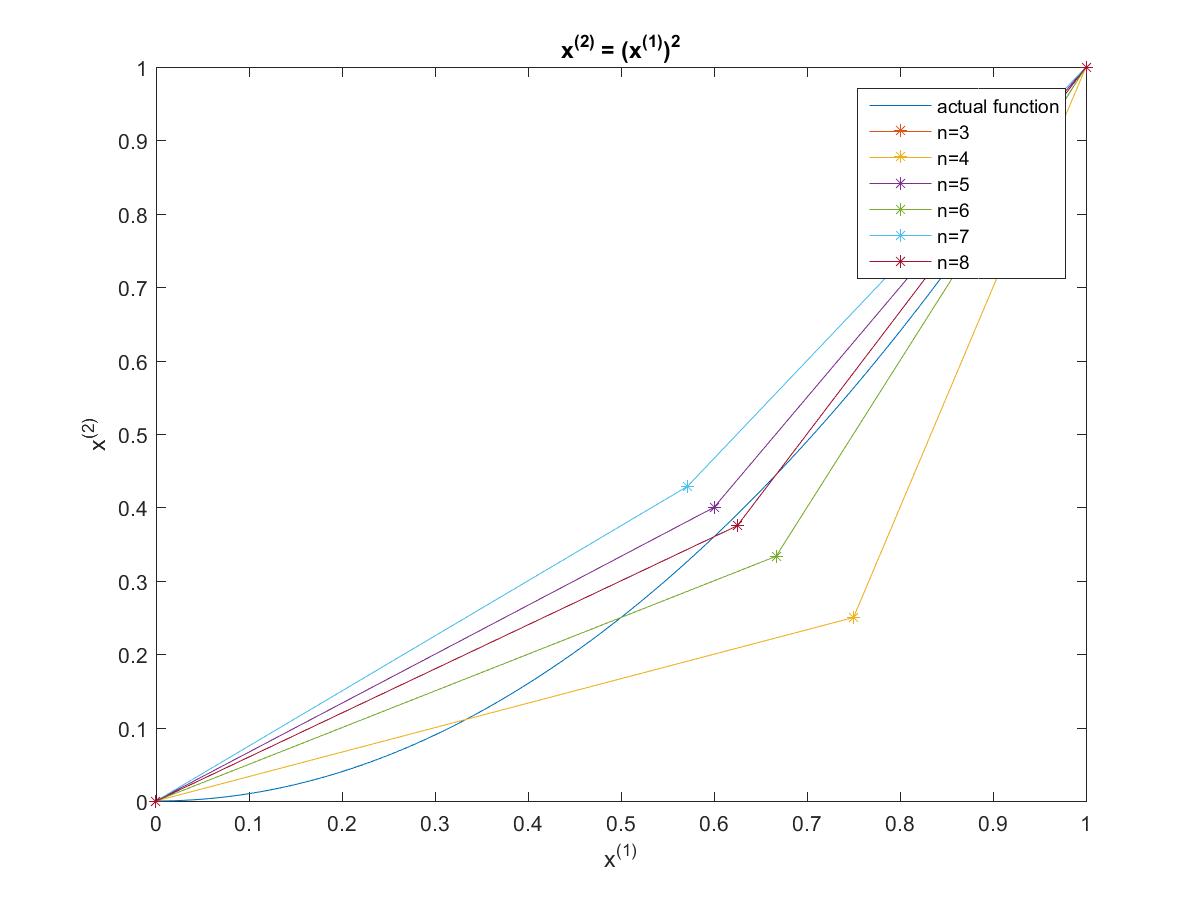}
	\caption{MLE estimators of $\gamma$ for $k=2$ and $3 \leq n \leq 8$}
	\label{figure:k2}
\end{figure}

\subsection{$n=4$, \quad $k=3,4$}

Here, we fix $n = 4$, and take $k=3$ or $4$. The weights for each approximation are listed in Tables \ref{table:3_4} and \ref{table:4_4} below, and the MLS estimator of the path are plotted in \ref{figure:n4}. Note that it also includes the case of $(k,n) = (2,4)$ from the previous subsection.

\begin{table}[!ht]
	\centering
	\begin{tabular}{c | c c c c c   }
		\hline
		$j \backslash m$  &0 &1&2&3&4\\
		\hline
		1&1.40E-02	&1.01E-01	&2.87E-01	&{\color{red}3.87E-01}	&2.11E-01\\
		2&8.84E-02	&2.89E-01	&{\color{red}3.65E-01}	&2.11E-01	&4.71E-02\\
		3&1.58E-01	&{\color{red}3.69E-01}	&3.26E-01	&1.28E-01	&1.90E-02\\
		\hline
	\end{tabular}
	\caption{$\mathcal{P}_{n, k}$ for $n = 3$, $k =4$}
	\label{table:3_4}
\end{table}

\begin{table}[ht]
	\centering
	\begin{tabular}{c | c c c c c   }
		\hline
		$j \backslash m$  &0 &1&2&3&4\\
		\hline
		1&8.78E-03 & 7.30E-02 & 2.47E-01 & {\color{red}4.02E-01} & 2.69E-01 \\
		2&6.35E-02 & 2.45E-01 & \color{red}{3.67E-01} & 2.55E-01 & 6.98E-02 \\
		3&1.16E-01 & 3.26E-01 &\color{red}{ 3.52E-01} & 1.74E-01 & 3.33E-02 \\
		4&1.67E-01 &\color{red}{3.76E-01} & 3.19E-01 & 1.20E-01 & 1.72E-02\\
		\hline
	\end{tabular}
	\caption{$\mathcal{P}_{n, k}$ for $n = 4$, $k =4$}
	\label{table:4_4}
\end{table}

\begin{figure}[!ht]
	\centering
	\includegraphics[width=0.8\textwidth]{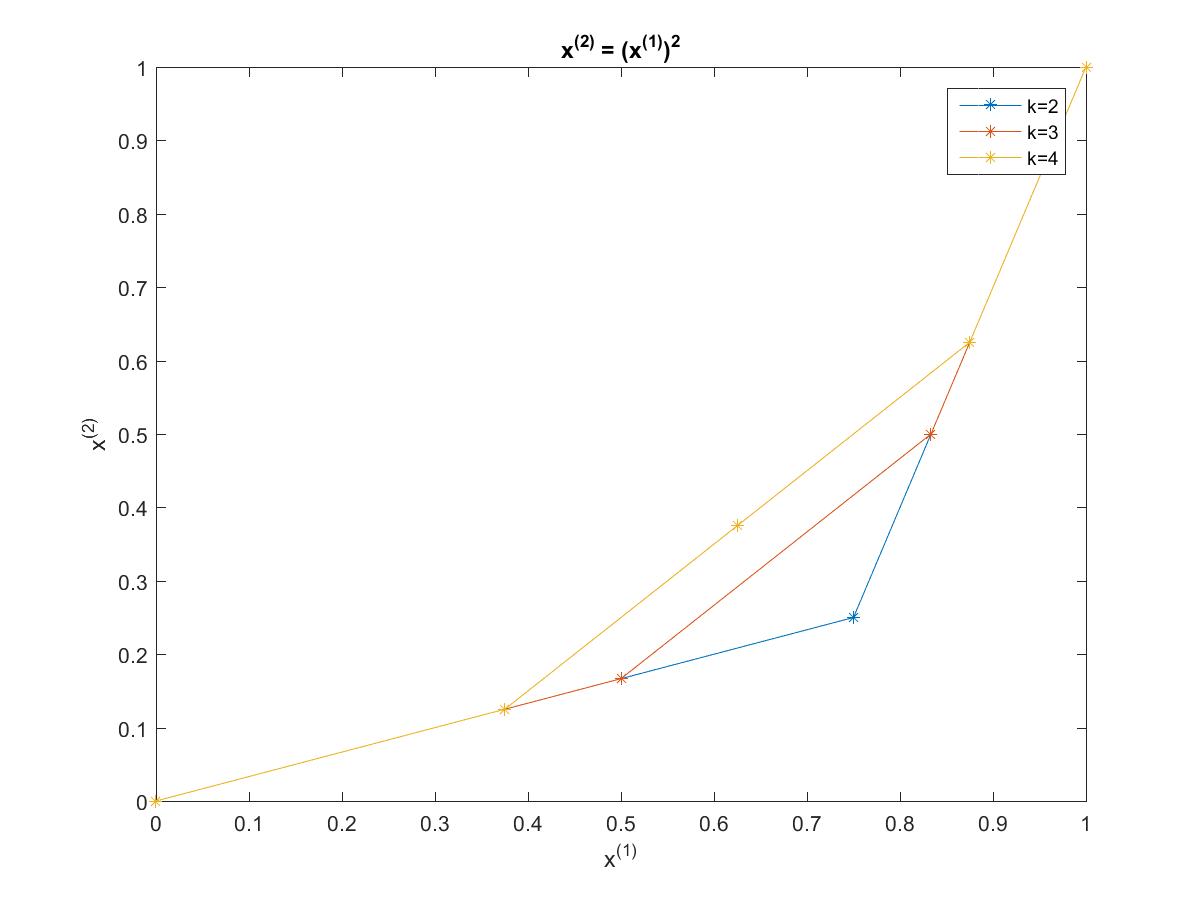}
	\caption{MLE estimators of $\gamma$ for $n=4$ and $k=2,3,4$}
	\label{figure:n4}
\end{figure}


\bibliographystyle{Martin}
\bibliography{Refs}

\end{document}